\documentclass{amsart}
\usepackage{amsmath,graphicx,amssymb,amsthm}

\newtheorem{theorem}{Theorem}
\newtheorem{proposition}{Proposition}
\newtheorem{lemma}{Lemma}

\theoremstyle{definition}
\newtheorem{definition}{Definition}
\newtheorem{corollary}{Corollary}

\newcommand{\abs}[1]{\left\vert#1\right\vert}
\newcommand{\set}[1]{\left\{#1\right\}}

\newcommand{\med}{\mathop{\vert}}

 \DeclareMathOperator{\pa}{PAut}
\DeclareMathOperator{\dom}{dom} \DeclareMathOperator{\ran}{ran}
 \DeclareMathOperator{\aut}{Aut}

\newcommand{\isd}{\mathcal{IS}_n}
\newcommand{\is}[1]{\mathcal{IS}({#1})}
\newcommand{\nd}{\mathcal{N}_n}

\newcommand{\spx}{S^{PX}}

\newcommand{\pwr}{\mathop{\wr_p}}

\newcommand{\wisd}{S \pwr \isd}
\newcommand{\fa}[1]{(f_{#1}, a_{#1})}

\newcommand{\lc}{\mathbin\mathcal{L}}
\newcommand{\rc}{\mathbin\mathcal{R}}
\newcommand{\hc}{\mathbin\mathcal{H}}
\newcommand{\dc}{\mathbin\mathcal{D}}

\newcommand{\1}{\mathbf{1}}
\newcommand{\0}{\mathbf{0}}

\newcommand{\sbs}{\subset}
\newcommand{\sbeq}{\subseteq}
\newcommand{\ra}{\rightarrow}
\newcommand{\vph}{\varphi}
\newcommand{\s}{\sigma}

\newcommand{\ld}{\{1, \ldots, n\}}
\newcommand{\bv}[1]{\big|_{#1}}

\begin{document}

\title{On cross-sections of partial wreath product of inverse semigroups}%
\author{EUGENIA KOCHUBINSKA}
\address{Taras Shevchenko National University of Kyiv, Faculty of Mechanics and Mathematics, Volodymyrska str. 64, 01601, Kyiv, Ukraine.}

\begin{abstract}
We classify  $\rc$- and $\lc$-cross-sections of partial wreath
product of inverse  semigroups. As a corollary,  we get the
description of $\rc$- and $\lc$-cross-sections of the semigroup $\pa
T$ of partial automorphisms of finite regular rooted tree $T$ and
compute also the number of different $\rc$- ($\lc$-) cross-sections
in this semigroup.
\end{abstract}


\maketitle

\section{Introduction}

Green's relations are basic relations introduced on a semigroup.
Therefore it is natural that the problem of description of
cross-sections of Green's relations has been arisen. During the last
decade cross-sections of Green's relations for some classical
semigroups were studied by different authors. In particular, all
$\hc$-cross-sections of inverse symmetric semigroup $\isd$ were
studied in \cite{CoReilly}. All $\rc$- and $\lc$-cross-sections were
classified in paper \cite{GM}.

In the present paper we describe all cross-sections of $\rc$ and
$\lc$ Green's relations of partial wreath product of inverse
semigroups. We also count the number of different $\rc$- and
$\lc$-cross-sections of this semigroup. The paper is organized as
follows. In Section 2 we collect all necessary basic definitions and
propositions. In Section 3 we provide a description of all $\rc$-
and $\lc$-cross-sections and compute the number of different $\rc$-
and $\lc$-cross-sections.

\section{Basic definitions}\label{sec:def}
For a set $X$, let $\mathcal{IS}(X)$ denote the set of all partial
bijections on $X$. On the set $\mathcal{IS}(X)$ define a composition
law: $ f \circ g: \dom(g \cap g^{-1} \dom(f))$ $f\circ g =g(f(x)),
\; x \in \dom(g \cap g^{-1}\dom(f))$, where $f, \; g \in
\mathcal{IS}(X)$. Under this operation set $(\mathcal{IS}(X),
\circ)$ forms a semigroup. This semigroup is called the \emph{full
inverse symmetric semigroup} on $X$. If $X=\nd$, where $\nd=\ld$,
then semigroup $\mathcal{IS}(\nd)$ is called the full inverse
symmetric semigroup of rank $n$ and is denoted $\isd$.

It is possible to introduce for elements of $\isd$ an analogue of
cyclic decomposition for elements of symmetric group
$\mathcal{S}_n$. We start with introducing two classes of elements.
Let $A=\{x_1, x_2, \ldots, x_k\} \sbs \nd$ be an ordered subset.
Denote by  $(x_1, x_2, \ldots, x_k)$ the unique element $f\in \isd$
such that $f(x_i)=x_{i+1}$, $i=1,2, \ldots, k-1$, $f(x_k)=x_1$ 
$f(x)=x, x \notin A$. Assume that $A\neq \emptyset$ and denote by
$[x_1, x_2, \ldots, x_k]$ the unique element $f\in \isd$ such that
$f(x_i)=x_{i+1}$, $i=1,2, \ldots, k-1$, $x_k \notin \dom(x)$ 
$f(x)=x$, $x \notin A$. The element $(x_1, x_2, \ldots, x_k)$ is
called a cycle and the element $[x_1, x_2, \ldots, x_k]$ is called a
chain. Any element of $\isd$ decomposes uniquely into the product of
disjoint cycles and chains. This decomposition is called a
\emph{chain decomposition} \cite{GM1}.

Recall the definition of partial wreath product of semigroups. Let
$S$ be a semigroup, $(P,X)$ be a semigroup of partial
transformations of the set $X$. Define the set $\spx$ as a set of
partial functions from $X$ to semigroup $S$:
\begin{equation*}
\spx=\{f:A \ra S| \dom(f)=A, A \sbeq X\}.
\end{equation*}
Given $f,g \in \spx$, the product $fg$ is defined in a following
way:
\begin{equation*}
\dom(fg)=\dom(f)\cap\dom(g), (fg)(x)=f(x)g(x) \text{\ for all } x\in
\dom(fg).
\end{equation*}
For $a\in P, f\in \spx$, define $f^a$ as:
\begin{equation*}
\begin{gathered}
(f^a)(x)=f(xa),\ \dom(f^a)=\{x \in \dom(a); xa\in \dom(f)\}.
\end{gathered}
\end{equation*}

\begin{definition}
\emph{Partial wreath product} of semigroup $S$ with semigroup
$(P,X)$ of partial transformations of the set $X$ is a set
$$\{(f,a)\in \spx \times (P,X)\,|\,\dom(f)=\dom(a) \}$$ with composition defined by
$(f,a)\cdot (g,b)=(fg^a,ab).$ We will denote partial wreath product
of semigroups $S$ and $(P,X)$ by $S \pwr P$.
\end{definition}

It is known \cite{Meldrum} that partial wreath product of semigroups
is a semigroup. Moreover, partial wreath product of inverse
semigroups is an inverse semigroup. An important example of inverse
semigroup is the semigroup $\pa T_n^k$ of partial automorphisms of a
$k$-level $n$-regular rooted tree $T_n^k$. By a partial automorphism
we mean a root-preserving tree homomorphism defined on a connected
subtree  of $T_n^k$. It is shown in \cite{comb} that $$\pa T_n^k
\simeq \underset{k}{\underbrace{\isd \pwr \isd \pwr \cdots \pwr
\isd}}.$$ This is an analogue of the well-known fact that $\aut
T_n^k \simeq \mathcal{S}_n\wr\dots \wr \mathcal{S}_n$.

\section{Description of $\rc$- and $\lc$-cross sections of semigroup $S\pwr \isd$}\label{sec:cross}
In this section we study cross-sections of partial wreath product of
finite inverse semigroup $S$ with semigroup $\isd$. Denote by $\0$
and $\1$ correspondingly the zero and the unit  of the semigroup
$S$.

Recall that Green's $\rc$-relation on inverse semigroup $H$ is
defined by $a\rc b \Leftrightarrow aH^1=bH^1$, similarly Green's
$\lc$-relation is defined by $a\lc b \Leftrightarrow H^1a=H^1b$.
Note that every $\rc$- ($\lc$-) equivalence class contains exactly
one idempotent. It is well-known (see for example \cite{GM1}) that
Green's relations on $\isd$ can be described as follows: $a \rc b
\Leftrightarrow \dom(a)=\dom(b)$; $a \lc b \Leftrightarrow
\ran(a)=\ran(b)$.

$\rc$- and $\lc$-relations on  $\wisd$ are described in the next
proposition.
\begin{proposition} \label{green_2}
\begin{enumerate}
\item $(f,a)$ $\rc$ $(g,b)$ if and only if
$\dom(a)=\dom(b)$ and for any $z \in \dom(a)$ $f(z)\rc g(z)$;

\item $(f,a)\lc(g,b)$ if and only if
$\ran(a)=\ran(b)$ and for any $z \in \ran(a)$ $g^{a^{-1}}(z)\lc
f^{b^{-1}}(z)$, where $a^{-1}$ is an inverse for $a$.
\end{enumerate}
\end{proposition}
\begin{proof}
The proof is completely analogous to the one for $\isd \pwr \isd$ in
\cite{comb}.
\end{proof}

Now let $\rho$ be an equivalence relation on a semigroup $H$. A
subsemigroup $T \sbs H$ is called \emph{cross-section with respect
to} $\rho$ provided that $T$ contains exactly one element from every
equivalence class. The cross-sections with respect to  $\rc$-
$(\lc$-) Green's relations are called $\rc$- ($\lc$-)
cross-sections. Note that every $\rc$- ($\lc$-) equivalence class
contains exactly one idempotent. Then the number of elements in
every cross-section is $\abs{E(H)}$, where $E(H)$ is the
subsemigroup of all idempotents of $H$.

It is not difficult to observe that a subsemigroup $H$ of semigroup
$\isd$ is an $\rc$-cross-section if and only if for every
subsemigroup $A\subseteq \nd$ it contains exactly one element $a$
such that $\dom(a)=A$.

Before describing $\rc$- and $\lc$-cross-sections in semigroup
$\wisd$, we recall first the description of $\rc$- and
$\lc$-cross-sections in semigroup $\isd$ presented in \cite{GM1}.
Let now  $\nd=M_1 \sqcup M_2 \ldots \sqcup M_s$~be an arbitrary
decomposition of $\nd=\{1,2, \ldots, n\}$ into disjoint union of
non-empty blocks, where the order of blocks is not important. Assume
that a linear order is fixed on the elements of every block: $M_i =
\set{m_1^i<m_2^i<\dots<m_{\abs{M_i}}^i}$.

For each pair $i,j$ $1 \leq i \leq k$, $1 \leq j \leq |M_i|$ denote
by $a_{i,j}$ the element in $\dc$-class $D_{n-1}$ of rank $n-1$ of
semigroup $\isd$, containing chain $[m_1^i, m_2^i, \ldots, m_j^i]$,
that acts as identity on the set $\nd\setminus \{m_1^i, m_2^i,
\ldots, m_j^i\}$. Denote by $R=R(\overrightarrow{M_1},
\overrightarrow{M_2}, \ldots, \overrightarrow{M_k})$ the semigroup
$\langle a_{i,j} |\; 1 \leq i \leq k, 1\leq j \leq |M_i|\rangle
\sqcup \{e\}$.

\begin{theorem}{\em\cite{GM}}
For an arbitrary decomposition  $\nd=M_1 \sqcup M_2 \ldots \sqcup
M_k$  and arbitrary linear orders on the elements of every block of
this decomposition the semigroup $R(\overrightarrow{M_1},
\overrightarrow{M_2}, \ldots, \overrightarrow{M_k})$ is an
$\rc$-cross-section of $\isd$. Moreover, every $\rc$-cross-section
is of the form  $R(\overrightarrow{M_1}, \overrightarrow{M_2},
\ldots, \overrightarrow{M_s})$ for some decomposition $\nd=M_1
\sqcup M_2 \ldots \sqcup M_s$ and some linear orders on the elements
of every block.
\end{theorem}

Since map $a \mapsto a^{-1}$ is an anti-isomorphism of semigroup
$\isd$ that sends  $\rc$-cross-sections to $\lc$-cross-sections,
then $\lc$-cross-sections is described similarly.

Now we turn to description of $\rc$- and $\lc$-cross-sections of
semigroup $\wisd$. It follows from Proposition \ref{green_2} that a
subsemigroup $H \sbs\wisd$ is an $\rc$-cross-section if and only if
for any $A \sbs \nd$ and any collection of idempotents $e_1, \ldots,
e_{|A|}\in E(S)$ there exists exactly one element $(f,a)\in H$
satisfying $\dom(a)=A$ and $f(x_i)\rc e_i$ for all $x_i \in A$.
Later we will use  this fact frequently.

We start the proof of the main result  with a sequel of lemmas.

\begin{lemma}\label{lem:l1} Let $R$ be an $\rc$-cross-section of semigroup
$\wisd$. Then  $$R_1=\{a \in \isd \med (f,a) \in R\}$$ is an
$\rc$-cross-section of semigroup $\isd$.
\end{lemma}
\begin{proof}
Let $(f,a)$, $(g,b)$ be elements of $\rc$-cross-section $R$, that
is, $a,b \in R_1$. The product of these elements
$(f,a)(g,b)=(fg^a,ab)$ is again in $\rc$-cross-section $R$, hence if
$a,b \in R_1$, then also $ab\in R_1$. Then $R_1$ is a semigroup.

If for element $(\varnothing, e)\in R$ we have $\dom(e)=\nd$, then
$\dom(e^2)=\dom(e)$, and since an $\rc$-cross-section contains only
one element, domain of which is equal $\nd$, then $e^2=e$, and hence
$e=id_{\nd}$. Then for every element $(f,a)\in R$ the product
$(f,a)(\varnothing, e)=(\varnothing, e)(f,a)=(\varnothing, a)\in R$.

As $R$ is an $\rc$-cross-section, then for every subset $A \sbs \nd$
there exists exactly one element $a\in R_1$ such that $\dom(a)=A$.
So, $R_1$ is an $\rc$-cross-section of semigroup $\isd$.
\end{proof}

\begin{lemma}\label{lem:l2}
Let $R$ be an $\rc$-cross-section of semigroup $\wisd$. Then
$$R_2=\{f(1) \med (f,a) \in R,  a=id_{\nd}, f(x)=\0, \text{\
for all\ } x\neq 1\}$$ is an $\rc$-cross-section of semigroup $S$.
\end{lemma}

\begin{proof}
Let $(f,a),(g,b) \in R$ be such that $a=b=id_{\nd}$, $f(x)=g(x)= \0,
\text{ for all } x \neq 1$, that is, $f(1), g(1) \in R_2$. Product
$(f,a)(g,b)=(fg^a, ab)$ satisfies condition $ab=id_{\nd}$,
$fg^a(x)=\0$ for all $x\neq 1$. Since $fg^a(1)\in R_2$ and
$fg^a(1)=f(1)g(1)$, then for $f(1), g(1) \in R_2$ their product also
belongs to $R_2$. Hence $R_2$ is a semigroup. As for every $f(1) \in
R_2$ the corresponding element $(f,a)$ is an element of an
$\rc$-cross-section of semigroup $\wisd$, then for every idempotent
$e \in E(S)$ exists exactly one element $f(1)\in R_2$ such that $e
\rc f(1)$. Thus $R_2$ is indeed an $\rc$-cross-section of semigroup
$S$.
\end{proof}

\begin{lemma} \label{lem:l3}
Let $S$ be an inverse semigroup, $\psi:S \ra S$ be an automorphism,
$R$ be an $\rc$-cross-section of $S$. Then $\psi(R)$ is also an
$\rc$-cross-section of $S$.
\end{lemma}

\begin{proof}
For any idempotent $e \in E(S)$ of semigroup $S$ there exists unique
element $a \in \psi(R)$ such that $aa^{-1}=e$. Let $a=\psi(b)$ for
$b\in R$. Since $\psi$ is an automorphism, then $\psi(bb^{-1})=e$ if
and only if $aa^{-1}=e$. An element $\psi^{-1}(e)$ is also an
idempotent of semigroup $S$. Then $bb^{-1}=\psi^{-1}(e)$ if and only
if $\psi(bb^{-1})=e$. The uniqueness of element $b$ such that
$bb^{-1}=\psi^{-1}(e)$ follows from the uniqueness of element $a$.
Thus $\psi(R)$ is an $\rc$-cross-section of $S$.
\end{proof}

\begin{lemma} \label{lem:l4}
Let $R$ be an $\rc$-cross-section of semigroup $S \pwr\is{M}$. If
$R_1=R(\overrightarrow{M})$, then $R \simeq R_2 \pwr R_1$.
\end{lemma}
\begin{proof}
The following holds for a partial wreath product $P=P_2\pwr P_1$ of
$\rc$-cross-sections.
\begin{equation}
\text{If  $(f,a)\in R$ is such that $f(i)\rc \1$ for $i\in \dom(a)$,
then $f(i)=\1$.}\label{1p}
\end{equation}  We will show now that the general case
when $f(i)\neq \1$ reduces to this one.

We may assume $M=\set{1,2, \ldots, m}$ with the usual order. Let
$(f_i, a_i)$, $i=1, \ldots, m$ be such elements of an
 $R$ that $a_i(1)=i$, $f_i(1)\rc \1$.  Put $\vph_i=f_i(1)$.

Consider now a map $\Theta: S\pwr\is{M} \ra S \pwr \is{M}$, which
acts as: $(f,a) \mapsto (g,a)$, where for $x \in \dom(a)$ we define
$g(x)=\vph_x f(x) \vph^{-1}_{x^a}$. It is easy
to check that this map is an isomorphism. 
From Lemma \ref{lem:l3} it follows that isomorphic image $\Theta(R)$
of an $\rc$-cross-section $R$ is an $\rc$-cross-section too.
Moreover, the next paragraph shows that for $\Theta(R)$ the property
\eqref{1p} is true.

Let $\psi_j=(f_{\psi_j}, b_{\psi_j})$ be such an element of
$\rc$-cross-section that $b_{\psi_j}(j)=m$, $f_{\psi_j}(j)\rc \1$.
Let $(f,a) \in R$ be an element such that $f(x)\rc \1$ for some $x
\in \dom(a)$. Consider now the product of elements $(f_x,a_x),
(f,a)$ and $(\psi_{x^a}, b_{x^a})$. We obtain
$(f_x,a_x)(f,a)(\psi_{x^a}, b_{x^a})=(f_x f^{a_x} \psi_{x^a}^{a_xa},
a_xab_{x^a})$. Domain of component $a_xab_{x^a}$ is the set $\{1\}$
and $a_xab_{x^a}(1)=m$. Also $f_x f^{a_x} \psi_{x^a}^{a_xa}\rc \1$,
and $(f_x f^{a_x}
\psi_{x^a}^{a_xa})(1)=f_x(1)f(1^{a_x})\psi_{x^a}(1^{a_xa})=f_x(1)f(x)\psi_{x^a}(x^a)$.
It is obvious that $(f_x,a_x)(f,a)(\psi_{x^a}, b_{x^a})=(f_m,a_m)$,
hence $(\vph_{x} f^{a_x} \psi_{x^a}^{a_{x}a})(1)=\vph_m$. Then we
have $(\vph_x f^{a_x} \vph_{x^a}^{-1}\vph_{x^a}
\psi_{x^a}^{a_xa})(1)=\vph_m$, but $\vph_{x^a}\psi{x^a}=\vph_m$.
Thus $g(x)=\vph_xf(x)\vph^{-1}_{x^a}=\1$.

As $R\simeq \Theta(R)$, we may assume that for $R$ itself the
property \eqref{1p} holds. In this case we will show $R=R_2\pwr
R_1$.

Let $\vph=\fa{\vph}$ be some element of $\rc$-cross-section $R$.
Then $a_{\vph} \in R_1$. We want to show that $f_{\vph}(i) \in R_2$
for arbitrary $i \in\dom(a_{\vph})$.

For that we put $j=a(i)$ and  define three groups of elements of
semigroup $S\pwr\is{M}$: element $\psi_i=\fa{\psi_i}$, where
$a_{\psi_i}=[1,i, i+1, \ldots, m-1, m]$, $f_{\psi_i}(1)=\1$,
$f_{\psi_j}=\0, j \geq i$; element $\s=\fa{\s}$, where
$a_{\s}=id_{M}$, and $f_\s(1)\rc f_{\vph}(i)$, and $f_{\s(1)}=\0$
when $x \neq 1$; element $\tau_j=\fa{\tau}$, where $a_{\tau_j}=[j,
m]$, $f_{\tau_j}(x)=\1$ for $x \in \dom(a)$. All of them are in $R$,
because they are the only possible elements for corresponding
domains and idempotents.

Consider product of elements $\psi_i$, $\vph$, and $\tau_j$. Then we
obtain $\psi_i \cdot \vph \cdot
\tau_j=(f_{\psi_i}f_{\vph}^{a_{\psi_i}}f_{\tau_j}^{a_{\psi_i}a_{\vph}},
a_{\psi_i}a_{\vph}a_{\tau_j})$. Domain of component
$a_{\psi_i}a_{\vph}a_{\tau_j}$ is the set $\{1\}$. Then
$\dom(f_{\psi_i}f_{\vph}^{a_{\psi_i}}f_{\tau_j}^{a_{\psi_i}a_{\vph}})=\dom(a_{\psi_i}a_{\vph}a_{\tau_j})=\{1\}$
 and
$({f_{\psi_i}f_{\vph}^{a_{\psi_i}}f_{\tau_j}^{a_{\psi_i}a_{\vph}}})(1)
\rc f_{\vph}(i)$.

For the product $\s \cdot \psi_{m}=(f_{\s}f_{\psi_{m}},
a_{\s}a_{\psi_{m}})$  we have that domain of $a_{\s}a_{\psi_{m}}$ is
the set $\{1\}$, then
$\dom(f_{\s}f_{\psi_{m}})=\dom(a_{\s}a_{\psi_{m}})=\{1\}$ and
$({f_{\s}f_{\psi_{m}}})(1)\rc f_{\vph}(i)$.

Thus we obtain that
$\dom(a_{\psi_i}a_{\vph}a_{\tau_j})=\dom(a_{\s}a_{\psi_{m}})$ and
$(f_{\psi_i}f_{\vph}^{a_{\psi_i}}f_{\tau_j}^{a_{\psi_i}a_{\vph}})(1)\rc(f_{\s}f_{\psi_{m}})(1)$.
Then
$(f_{\psi_i}f_{\vph}^{a_{\psi_i}}f_{\tau_j}^{a_{\psi_i}a_{\vph}})(1)=(f_{\s}f_{\psi_{m}})(1)$,
because $R$ is $\rc$-cross-section.

As element $f_{\s}(1)$ lays in $\rc$-cross-section $R_2$, then also
$f_{\s}f_{\psi_{m}}= f_{\s}(1)f_{\psi_{m}}(1)$ is in $R_2$, because
$f_{\psi_{m}}=\1$. Then product
$f_{\psi_i}f_{\vph}^{a_{\psi_i}}f_{\tau_j}^{a_{\psi_i}a_{\vph}}=f_{\psi_1}(1)f_{\vph}(i)f_{\tau_j}(j)$
also in $R_2$. Since $f_{\psi_i}(1)=f_{\tau_j}(j)=\1$, then
$f_{\vph}(i)\in R_2$. Concluding we have $R \sbs R_2\pwr R_1$.

The number of elements of $\rc$-cross-section of inverse semigroup
is equal to the number of idempotents of this semigroup. The element
$(f,a)$ of the semigroup $S\pwr\mathcal{IS}(M)$ is idempotent iff
all $a$ and $f(i)$ are idempotents. Then number of idempotents of
this wreath product equals $(\abs{E(S)}+1)^{m}$. The number of
elements of partial wreath product $R_2 \pwr R_1$ equals $\sum
_{i=1}^{m} \vert R_2 \vert^i \cdot \binom{m}{i}=(\abs{E(S)}+1)^{m}$.
Therefore $R=R_2 \pwr R_1$.

\end{proof}

\begin{theorem} \label{thm:rcross}
Let $R(\overrightarrow{M_1}, \overrightarrow{M_2}, \ldots,
\overrightarrow{M_k})$ be $\rc$-cross-section of semigroup $\isd$,
$R_1, \ldots, R_k$ be $\rc$-cross-sections of semigroup $S$. Then
$$R=(R_1\pwr R(\overrightarrow{M_1}) )\times (R_2 \pwr
R(\overrightarrow{M_2}) ) \times \ldots \times (R_{k}\pwr
R(\overrightarrow{M_k}) )$$ is an $\rc$-cross-section of semigroup
$S \pwr \isd$. Moreover, every $\rc$-cross-section is isomorphic to
$(R_1 \pwr R(\overrightarrow{M_1}) )\times (R_2 \pwr
R(\overrightarrow{M_2}) ) \times \ldots \times (R_{k}\pwr
R(\overrightarrow{M_k}) )$. \end{theorem}
\begin{proof}
 Let  $R_1,
\ldots, R_{k}$ be $\rc$-cross-sections of semigroup $S$. It is
obvious that $(R_1 \pwr R(\overrightarrow{M_1}) )\times (R_2 \pwr
R(\overrightarrow{M_2}) ) \times \ldots \times (R_{k}\pwr
R(\overrightarrow{M_k}) )$ is a semigroup.

Let $h=\fa{h}$ be an element of $\wisd$. Now show that there exists
only one element $g=\fa{g}\in R$ such that $h \rc g$. Define $g$ in
a following way. Put $a_g\bv{M_i}=b_i$, where $b_i \in
R_0(\overrightarrow{M_i})$, $\dom(b_i)=\dom(a _h)\cap M_i$.  For
every $x_i\in M_i\cap \dom(a_h)$ put $f_g(x_i)=y_i$, where $y_i\in
R_i$, $y_i\rc h(x_i)$. It follows from definition of $g$ that $g\rc
h$. It is clear that such an element $g$ is unique.

Now prove that every $\rc$-cross-section is obtained in this way.
Let $R$ be an $\rc$-cross-section of semigroup $\wisd$. According to
Lemma 1, the set $R_1=\{a| (f,a) \in R\}$ is an $\rc$-cross-section
of semigroup $\isd$, hence $R_1=R(\overrightarrow{M_1},
\overrightarrow{M_2},\ldots, \overrightarrow{M_s})$ for some
decomposition $M_1 \sqcup M_2 \sqcup \ldots \sqcup M_s$ of $\nd$.

Let $(g_i,e_i)\in R$ be such that $\dom(e_i)=M_i$, $g_i(x)\rc\1$ for
all $x \in M_i$.  Then analogously to Lemma 1, $e_i=id_{M_i},
g_i(x)=\1$. This element is the element of $\rc$-cross-section $R$.

As $M_i^R=M_i$ and $\nd=M_1\sqcup M_2 \sqcup \ldots \sqcup M_k$ we
have monomorphism from $R$ to $\prod_{i=1}^k(S \pwr \is{M_i})$
defined by $(f,a) \mapsto \left((f\bv{M_1}, a\bv{M_1}), \ldots,
(f\bv{M_k},a\bv{M_k})\right)$. Similarly to Lemma 1, multiplying by
$(g_i, e_i)$, we get that each component $R(g_i,e_i)$ of image of
$R$ is  an $\rc$-cross-section of $S \pwr \is{M_i}$. From Lemma 4 it
follows that every component is isomorphic to $R_i\pwr
R(\overrightarrow{M_i}) $ for some $\rc$-cross-section $R_i$ of
semigroup $S$. Now the statement of theorem is obvious.
\end{proof}

A map $a \mapsto a^{-1}$ is an anti-isomorphism of semigroup
$\wisd$, that sends $\rc$-classes to $\lc$-classes. It is also clear
that it maps $\rc$-cross-sections to $\lc$-cross-sections and
vice-versa. Hence dualizing Theorem \ref{thm:rcross}, one gets
description of $\lc$-cross-sections.

\begin{corollary}
Let $R(\overrightarrow{M_1}, \overrightarrow{M_2}, \ldots,
\overrightarrow{M_k}),R_1, \ldots, R_k$ be $\rc$-cross-sections of
semigroup $\isd$. Then $R=(R_1\pwr R(\overrightarrow{M_1}) )\times
(R_2 \pwr R(\overrightarrow{M_2}) ) \times \ldots \times (R_{k}\pwr
R(\overrightarrow{M_k}) )$ is an $\rc$-cross-section of semigroup
$\isd \pwr \isd$. Moreover, every $\rc$-cross-section is isomorphic
to $(R_1 \pwr R(\overrightarrow{M_1}) )\times (R_2 \pwr
R(\overrightarrow{M_2}) ) \times \ldots \times R_{k}\pwr
(R(\overrightarrow{M_k}) )$.
\end{corollary}
Starting from the last corollary and iterating
Theorem~\ref{thm:rcross}, one gets full description of
$\rc$-cross-sections of the semigroup $\pa T_n^k$ of partial
automorphisms of a rooted tree.
\begin{corollary}
Semigroup $\isd\pwr\isd$ contains $\sum\limits_{k=1}^n
\frac{(n!)^{n+1}}{k!}\binom{n-1}{k-1}\left(\sum\limits_{i=1}^n\frac{1}{i!}\binom{n-1}{i-1}\right)^k$
different $\rc$-($\lc$-) cross-sections.
\end{corollary}

\end{document}